\documentclass{amsart}
\usepackage{amssymb}
\usepackage{eucal}
\usepackage{enumerate}

\newcommand{\loc}{\textup{loc}}

\renewcommand{\H}{\mathcal{H}}

\newcommand{\R}{\mathbb{R}}

\newtheorem{thm}{Theorem}[section]
\newtheorem{prop}[thm]{Proposition}
\newtheorem{cor}[thm]{Corollary}
\newtheorem{lem}[thm]{Lemma}
\theoremstyle{definition}

\newtheorem{xmp}[thm]{Example}
\theoremstyle{remark}
\newtheorem{rem}[thm]{Remark}
\numberwithin{equation}{section}

\newcommand{\dd}{\partial}
\renewcommand{\d}{\,\textup{d}}

\newcommand{\supp}{\operatorname{supp}}
\newcommand{\dist}{\operatorname{dist}}
\newcommand{\lap}{\Delta}

\renewcommand{\div}{\operatorname{div}}
\newcommand{\osc}{\operatornamewithlimits{osc}}

\renewcommand{\Re}{\operatorname{Re}}
\renewcommand{\H}{\mathcal{H}}

\def\Xint#1{\mathchoice
{\XXint\displaystyle\textstyle{#1}}%
{\XXint\textstyle\scriptstyle{#1}}%
{\XXint\scriptstyle\scriptscriptstyle{#1}}%
{\XXint\scriptscriptstyle\scriptscriptstyle{#1}}%
\!\int}
\def\XXint#1#2#3{{\setbox0=\hbox{$#1{#2#3}{\int}$}
\vcenter{\hbox{$#2#3$}}\kern-.5\wd0}}

\def\dashint{\Xint-}

\title{The two-phase fractional obstacle problem}
\author{Mark Allen}
\address{Department of Mathematics, The University of Texas at Austin,
Austin, TX 78712, USA}
\email{mallen@math.utexas.edu}
\author{Erik Lindgren}
\address{Department of Mathematics, Royal Institute of Technology, 100 44 Stockholm, Sweden}
\email{eriklin@kth.se}

\author{Arshak Petrosyan}
\address{Department of Mathematics, Purdue University, West Lafayette,
IN 47907, USA}
\email{arshak@math.purdue.edu}
\thanks{M.A.\ and A.P.\ are partially supported by NSF grant
  DMS-1101139. E.L.\ thanks the Royal Swedish Academy of Sciences for partial support.}
\begin{document} 
\begin{abstract}
We study minimizers of the functional 
$$
\int_{B_1^+}|\nabla u|^2 x_n^a\d x +2\int_{B_1'} (\lambda_+
u^++\lambda_- u^-)\d x',
$$
for $a\in(-1,1)$.
The problem arises in connection with heat flow with control on the
boundary. It can also be seen as a non-local analogue of the, by now
well studied, two-phase obstacle problem. Moreover, when $u$ does not
change signs this is equivalent to the fractional obstacle problem.
Our main results are the optimal regularity of the minimizer and
the separation of the two free boundaries $\Gamma^+=\partial'\{u(\cdot,0)>0\}$ and
$\Gamma^-=\partial'\{u(\cdot,0)<0\}$ when $a\geq 0$.
\end{abstract}

\maketitle

\section{Introduction and motivation}

\subsection{The problem}
The main purpose of this paper is to study the minimizers of the
energy functional
\begin{equation}\label{eq:jfunc}
J_a(u)=\int_{D^+}|\nabla u|^2 x_n^a \d x +2\int_{D'}(\lambda_+
u^++\lambda_- u^-)\d x',
\end{equation}
where $D$ is a bounded domain in $\R^n$, 
$$
D^+:= D\cap\{x_n>0\},\quad D':=D\cap\{x_n=0\},
$$
$a\in(-1,1)$, $\lambda_\pm>0$, and
$u^\pm=\max\{\pm u,0\}$.
More precisely, we look at the functions $u$ in the weighted Sobolev space
$W^{1,2}(D^+,x_n^a)$, with prescribed boundary values 
\begin{equation}\label{eq:jfunc-bdry}
u=g\quad\text{on }(\partial D)^+:=\partial D\cap\{x_n>0\},
\end{equation}
in the sense of traces. This leads to a two-phase problem on $D'$ if we
identify the regions
$$
\Omega^+=\Omega^+_u:=\{u>0\}\cap D',\quad \Omega^-=\Omega^-_u:=\{u<0\}\cap D'
$$
as the phases of the minimizer $u$. We call their boundaries in $D'$
$$
\Gamma^+=\Gamma^+_u:=\partial'\Omega^+\cap D',\quad \Gamma^-=\Gamma^-_u:=\partial'\Omega^-\cap D'
$$
\emph{free boundaries}, since they are apriori unknown ($\partial'$
here stands for the boundary in $\R^{n-1}$). The
free boundaries of this type are known as \emph{thin free boundaries}
in the literature, since they are formally of co-dimension
two. Particular questions of interest are the regularity of the
minimizers and the regularity and structure of their free boundaries.

The problem above is strongly related to the so-called 
obstacle problem for fractional Laplacian $(-\Delta_{x'})^s$, $0<s<1$,
or more precisely its localized version. The latter
problem consists of minimizing  the
energy functional
\begin{equation}\label{eq:efunc}
E_a(v)=\int_{D^+} |\nabla v|^2x_n^a\d x,\quad a=1-2s
\end{equation}
among all functions $v\in W^{1,2}(D^+,x_n^a)$ satisfying
\begin{alignat}{2}
\label{eq:efunc-bdry}v&=h&\quad&\text{on }(\partial D)^+,\\
\label{eq:efunc-obst}v&\geq \phi &&\text{on }D'.
\end{alignat}
The function $\phi$ above is know as the \emph{thin obstacle}. In the
particular case of $a=0$, this problem is the scalar version of the
Signorini problem, also know as the \emph{thin} (or \emph{boundary})
\emph{obstacle problem}. The free boundary in this problem is
$$
\Gamma_{v}=\partial'\{v(\cdot,0)>\phi\}\cap D'.
$$
There has been a surge of interest in such
problems in recent years, because of the development of new tools such
as the Caffarelli-Silvestre extension, Almgren-type monotonicity formulas,
etc, that allowed significant advances in the study of the
regularity properties of the free boundaries; see \cite{AC04,Sil07,ACS07,CSS08,gp09}.

The exact connection between the minimizers of
\eqref{eq:jfunc}--\eqref{eq:jfunc-bdry} and
\eqref{eq:efunc}--\eqref{eq:efunc-obst} is given in the following
(relatively straightforward) 
proposition.

\begin{prop}\label{prop:noneg-one-phase} Let $u$ be a minimizer of
  \eqref{eq:jfunc}--\eqref{eq:jfunc-bdry} such that $u\geq 0$ on
  $D'$. Then $v(x)=u(x)-\frac{\lambda_+}{1-a} x_n^{1-a}$ is a
  minimizer of \eqref{eq:efunc}--\eqref{eq:efunc-obst} with
  $h(x)=g(x)-\frac{\lambda_+}{1-a} x_n^{1-a}$ and $\phi=0$.
\end{prop}
\begin{proof}
The proof is an immediate corollary of the following computation:

\begin{align*}
E_a(v)&=\int_{D^+} |\nabla u -\lambda_+
x_n^{-a}e_n|^2 x_n^a\\
&=\int_{D^+}(|\nabla u|^2x_n^a+\lambda_+^2
x_n^{-a}-2\lambda_+\partial_{x_{n-1}}u)\\
&=\int_{D^+}|\nabla u|^2x_n^a+2\int_{D'}(\lambda_+ u^+)+ C(D^+, g)\\
&=J_a(u)+C(D^+,g).\qedhere
\end{align*}
\end{proof}
In particular, when $u\geq 0$ we see that $u(x',0)\equiv v(x',0)$ and
consequently
$$
\Gamma^+_u=\Gamma_v,
$$
reducing the study of the free boundary in
\eqref{eq:jfunc}--\eqref{eq:jfunc-bdry} to that in the fractional
obstacle problem  \eqref{eq:efunc}--\eqref{eq:efunc-obst}. For the reasons above, we call the minimization problem \eqref{eq:jfunc}--\eqref{eq:jfunc-bdry} the
\emph{two-phase fractional obstacle problem}. When $a=0$ we also call it the 
\emph{two-phase thin obstacle problem}.

The two-phase fractional obstacle problem can also be viewed as a 
nonlocal version of the so-called \emph{two-phase obstacle problem} (also known as the two-phase membrane
problem)
\begin{equation}\label{eq:2phase-obst}
\Delta u=\lambda_+\chi_{\{u>0\}}-\lambda_-\chi_{\{u<0\}}\quad\text{in }D.
\end{equation}
Solutions of \eqref{eq:2phase-obst} can be
obtained by minimizing the the energy functional
$$
\tilde J(u)=\int_{D}(|\nabla u|^2+2\lambda_+ u^++2\lambda_- u^-)\d x.
$$
This
problem has been studied in detail in a series of papers \cite{Wei01,Ura01,SUW-global,
SW,SUW}. In particular, it is known that the
minimizers are $C^{1,1}$-regular and that the free boundaries $\tilde
\Gamma^\pm:=\partial\{u>0\}\cap D$ are graphs of $C^1$ functions near
two-phase points (i.e.\ points on $\tilde \Gamma^+\cap
\tilde\Gamma^-$). Away from two-phase points, the problem is locally
equivalent to the classical obstacle problem, which is extensively
studied in the literature. We also refer to the book
\cite{PSU12} for the introduction to this and other obstacle-type
problems.

\subsection{Main results} Our first result concerns the regularity of the minimizers.

\begin{thm}[Optimal regularity]\label{thm:main-opt-reg} Let $u$ be a minimizer of
  \eqref{eq:jfunc}. Then for any  $K\Subset D^+\cup D'$ we have
\begin{alignat*}{2}
u&\in C^{0,1-a}(K),&\quad&\text{if }a\geq 0\\
u&\in C^{1,-a}(K), &\quad&\text{if }a< 0
\end{alignat*}
with bounds on their respective norms, 
depending only on $K$, $D$, $a$, $\lambda_\pm$, $n$, and $\|u\|_{L^2(D^+)}$.
\end{thm}

Note that this regularity is sharp, because of the explicit example
$$
u(x)=\frac{\lambda_+}{(1-a)}x_n^{1-a},
$$
see Proposition~\ref{prop:noneg-one-phase}. 

As a by-product we also obtain Theorem \ref{thm:main1-EL}, which, by the Caffarelli-Silvestre extension, provides sharp regularity estimates for solutions of
$$(-\lap)^s u=f\in L^\infty,$$
if $s=(1-a)/2 \neq 1/2.$

Our second result is about the structure of the free boundary, which
effectively reduces the problem to the (one-phase) fractional obstacle problem in the case $a\geq 0$.

\begin{thm}[Structure of the free boundary]\label{thm:main-free-bound} Let $a \geq 0$ and let $u$ be a minimizer of \eqref{eq:jfunc}. Then
$$
\Gamma^+\cap \Gamma^-=\emptyset.
$$
Thus, by Proposition~\ref{prop:noneg-one-phase}, the study of the free boundary is reduced to the fractional
obstacle problem. 
\end{thm}

\subsection{Structure of the paper}
\begin{enumerate}[$\bullet$]
\item In Section~\ref{sec:notation}, we describe the notation used
throughout the paper. 
\item Section~\ref{sec:motivation} contains a couple
of applications that motivated our study. 
\item In
Section~\ref{sec:preliminaries} we collected some results on
$a$-harmonic functions: interior estimates, Liouville-type
theorem and Almgren's frequency formula,  as well as some preliminary
results on our problem: energy inequality, local boundedness,
Weiss-type monotonicity formula (the proof of the latter is
given in Appendix). It also contains the
Alt-Caffarelli-Friedman monotonicity formula to be used in the
case $a=0$.
\item In Section~\ref{sec:regul-minim} we prove some preliminary
  regularity on the minimizers: namely $C^{0,(1-a)/2}$, when $a\geq 0$, and
  $C_{x'}^{1,\alpha}$ for $\alpha<-a$, when $a<0$.
\item Section~\ref{sec:optim-growth-regul} contains the proof of
  Theorem~\ref{thm:main-opt-reg}. We first show the optimal growth
  rate of the minimizers from the free boundary
  (Theorem~\ref{thm:main1-growth}) which then implies the optimal
  regularity (Theorem~\ref{thm:main1-expanded}). 
\item In Section~\ref{sec:nondegeneracy} we prove the nondegeneracy of
  each phase near their respective free boundaries.
\item In Section~\ref{sec:blow-glob-solut} we study the blowups, by
showing first that they are homogeneous of degree $(1-a)$ and then
classifying all such minimizers for $a\geq 0$.
\item Finally, in Section~\ref{sec:struct-free-bound} we prove
  Theorem~\ref{thm:main-free-bound}. In fact, we prove its stronger
  version (Theorem~\ref{thm:main2-expanded}) that the free boundaries $\Gamma^+$ and $\Gamma^-$ are
  uniformly separated when $a\geq 0$. We conclude the paper with an example of a
  solution when $a<0$ where $\Gamma^+$ and $\Gamma^-$ have a common
  part with vanishing thin gradient $|\nabla_{x'}u|$.
\end{enumerate}

\section{Notation}\label{sec:notation}
Throughout the paper we will be using the following notation:
\begin{enumerate}[--]
\item  We denote a point $x \in \R^n$ by $(x',x_n)$ where $x'=(x_1,
  \ldots , x_{n-1})\in\R^{n-1}$. Moreover, in certain cases we will identify 
  $x'\in\R^{n-1}$ with $(x',0)\in\R^{n-1}\times\{0\}\subset \R^n$. We
  will refer to $\R^{n-1}=\R^{n-1}\times\{0\}$ as the \emph{thin space}.

\item $\nabla u=(\dd_{x_1}u,\ldots,\dd_{x_{n-1}}u,\dd_{x_n}u)$: the
  full gradient.

\item $\nabla_{x'} u=(\dd_{x_1}u,\ldots,\dd_{x_{n-1}}u)$: the gradient
  in the thin space, or the thin gradient.
\item $u_{x_i}$: alternative notation for  the partial derivative $\partial_{x_i} u$, $i=1,\ldots, n$. 

\item  For a set $E \subset \R^n$ we define
\[
E':= E \cap (\R^{n-1} \times \{0\})
\]
\item For $E\subset \R^{n-1}=\R^{n-1}\times\{0\}$, $\partial'E$ is the
  boundary in the topology of $\R^{n-1}$.

\item For a function $u$, we define $u^+ = \max\{u,0\}$ and $u^-= \max\{-u,0\}$, the positive and negative parts of $u$ so that, $u = u^+ - u^-$.

\item We call $\Gamma^+$ and $\Gamma^-$  the free boundaries where
\[
\Gamma^+ = \partial' \{u(\ \cdot \ ,0) > 0\} \text{ and } \Gamma^- = \partial' \{u(\ \cdot \ , 0) < 0\} 
\]

\item We will for $x_0\in \Gamma^+\cup\Gamma^-$, denote by $u_{r,x_0}$, the rescaled function
$$
\frac{u(rx+x_0)}{r^{1-a}}.
$$
If there is a subsequence $u_{r_j,x_0}$, where $r_j\to 0$, converging locally uniformly to a function $u_{0,x_0}$ we will say that $u_{0,x_0}$ is a \emph{blow-up} of $u$ at the point $x_0$.

\item $L^2(D^+,x_n^a)$ \big[respectively $L^2(D, |x_n|^a)$\big], for $|a|<1$, is the
  weighted Lebesgue space with the square of the norm given by 
$$
\int_{D^+} |f(x)|^2x_n^a \d x \quad\left[\text{respectively } \int_{D} |f(x)|^2|x_n|^a \d x\right].
$$

\item $W^{1,2}(D^+,x_n^a)$ \big[respectively  $W^{1,2}(D,|x_n|^a)$\big], for $|a|<1$, is the weighted Sobolev
  space of functions $u\in L^2(D^+, x_n^a)$ \big[respectively $L^2(D, |x_n|^a)$\big] such that the distributional derivatives $\partial_{x_i} u$, $i=1,\ldots, n$, are
  also in $L^2(D^+, x_n^a)$  \big[respectively $L^2(D,
  |x_n|^a)$\big]. For some of the basic results for these spaces such
  as the Sobolev embedding and Poincar\'e inequalities, we refer to
  \cite{FKS82}. See also \cite{Kil}.

It is known that if there is a bi-Lipschitz transformation that will take
$(D, D')$ to $(B_1, B_1')$ then functions in $W^{1,2}(D^+,x_n^a)$ \big[or
$W^{1,2}(D,|x_n|^a)$\big] have traces in $L^2(D')$ and $L^2((\partial
D)^+, x_n^a)$ \big[or $L^2(\partial
D, |x_n|^a)$\big]. Furthermore, the corresponding trace operators are
compact. See \cite{Nek} for the first trace operator. Compactness of
the second one follows from the approximation by Lipschitz
functions as indicated in Remark~6 in \cite{Kil}, similarly to the proof of Theorem~8 there.

\item We say that $u\in W^{1,2}(D,|x_n|^a)$ is \emph{$a$-harmonic} in an open set $D\subset \R^n$ if it is a weak solution of 
$$
L_a u:=\div (|x_n|^a \nabla u )=0.
$$
Equivalently, $a$-harmonic functions are the minimizers of the
weighted Dirichlet integral
$$
E_a(u):=\int_{D}|\nabla u|^2|x_n|^a\d x,
$$
among all functions with the same trace on $\partial D$.

\item For $u\in W^{1,2}(D,|x_n|^a)$, by an \emph{$a$-harmonic replacement} of
   $u$ on $U\Subset D$ we understand a function $v$ coinciding
  with $u$ on $D\setminus U$, $a$-harmonic in $U$, with the same boundary values on $\partial U$ as $u$
  in the sense $v-u\in W^{1,2}_0(U,|x_n|^a)$.

\item Very often we will use the following convention for
  integrals. If the measure of integration is not indicated then it is
  either the standard Lebesgue measure $\d x$ in $\R^n$, or
  $(n-1)$-dimensional Hausdorff measure $\d \H^{n-1}$ (also denoted
  $\d x'$ when restricted to $\R^{n-1}$). 

\end{enumerate}

\section{Motivation}\label{sec:motivation}
In this section we describe some applications that motivated us to
study the problem  \eqref{eq:jfunc}--\eqref{eq:jfunc-bdry}, in
addition to those indicated in the introduction.

\subsection{Temperature control problems}
One of the applications of the two-phase obstacle problem
\eqref{eq:2phase-obst} is the
temperature control problem through the interior, see Duvaut-Lions \cite[Chapter I, \S2.3.2]{dl76}. More specifically,
suppose that we want to maintain the temperature $u(x)$ in the domain $D$
close to the range $(\theta_-,\theta_+)$, where $\theta_-(x)\leq
\theta_+(x)$ are two given functions in $D$. For that purpose we have
an array of cooling/heating devices that are distributed evenly
over the domain $D$. The devices are assumed to be of limited power,
with their generated heat flux $-f$ in the range
$[-\lambda_-,\lambda_+]$, $\lambda_\pm>0$. The device will generate no heat if $u(x)\in
[\theta_-(x),\theta_+(x)]$, however, outside that range
the corrective heat flux $F$ will be injected according to the following rule:
\begin{align*}
F=F(u)=\begin{cases}
\min\{k_+(u-\theta_+),\lambda_+\},& u>\theta_+,\\
\max\{k_-(u-\theta_-),-\lambda_-\}, & u<\theta_-,
\end{cases}
\end{align*}
where $k_\pm\geq 0$. In the equilibrium state, the temperature distribution will satisfy
$$
\Delta u=F(u)\quad\text{in }D.
$$
Assuming $\theta_\pm=0$ and $k_\pm=+\infty$, the equation becomes
$$
\Delta u=\lambda_+\chi_{\{u>0\}}-\lambda_-\chi_{\{u<0\}}\quad\text{in }D,
$$
which is nothing but the two-phase obstacle problem.

Suppose now that we want to regulate the temperature in $D^+=D\cap\{x_n>0\}$ with the heating/cooling devices that are evenly distributed along the ``wall'' $D'=D\cap\{x_n=0\}$. Then in the equilibrium position, the temperature will satisfy
$$
\Delta u=0\quad\text{in }D^+,\quad \partial_{x_n} u=F(u)\quad\text{on }D'.
$$
In the limiting case $\theta_\pm=0$ and $k_\pm=\infty$ this transforms
into the variational inequality 
$$
\Delta u=0\quad\text{in }D,\quad \partial_{x_n} u\in\bar F(u)\quad\text{on }D',
$$
where $\bar F$ is the maximal monotone graph
\begin{equation}\label{eq:max-mon}
\bar F(u)=\begin{cases}
\{\lambda_+\},& u>0,\\
[-\lambda_-,\lambda_+],& u=0,\\
\{-\lambda_-\}, & u<0.
\end{cases}
\end{equation}
Equivalently, the solutions of this variational inequality can be obtained by minimizing the energy functional
$$
J_{0}(u)=\int_{D^+} |\nabla u|^2\d x+2\int_{D'} (\lambda_+ u^++\lambda_-
u^-) \d x',
$$
subject to a Dirichlet boundary condition on $(\partial D)^+$ (cf.\ Lemma~\ref{lem:normalbound}). This is a particular case ($a=0$) of the problem that we intend to
study in this paper. 

\subsection{Two-phase problem for fractional Laplacian}
Suppose now that $\Sigma$ is a bounded open set in $\R^{n-1}$,  $0<s<1$, and we
consider the minimizers of the energy functional
$$
j_s(v)=c_{n,s}\int_{\R^{n-1}}\int_{\R^{n-1}}\frac{(v(x')-v(y'))^2}{|x'-y'|^{n-1+2s}}+2\int_{\R^{n-1}}
(\lambda_+ u^++\lambda_-u^-),
$$
among all function such that $v=g$ on $\R^{n-1}\setminus \Sigma$.
Here $c_{n,s}>0$ is a normalization constant.
Then it is easy to see that the minimizers will satisfy the
variational inequality
\begin{equation}\label{eq:frac-var-ineq}
-(-\Delta_{x'})^s v\in \bar F(v)\quad\text{on }\Sigma,
\end{equation}
with the maximal monotone graph $\bar F$ in \eqref{eq:max-mon}.
To see the connection with the two-phase thin obstacle problem, we use
the Caffarelli-Silvestre extension \cite{CS07} for $v$. Namely, let $u$ be a
function $\R^n_+$ that satisfies
\begin{align*}
L_a u:=\div(x_n^a \nabla u)=0&\quad\text{in }\R^{n}_+\\
u(x',0)=v(x')&\quad\text{for }x'\in\R^{n-1},
\end{align*}
where $a=1-2s\in(-1,1)$. (The initial condition is understood in the
boundary trace sense.)
Then we can recover the fractional Laplacian of $v$ on $\R^{n-1}$ by
$$
(-\Delta_{x'})^sv(x')=C_{n,s}\lim_{x_n\to 0+}x_n^a\partial_{x_n} u(x',x_n),
$$
in the weak sense. Suppose now that there is a smooth domain $D$ in $\R^n$ is such that
$\Sigma=D'=D\cap\{x_n=0\}$. Then, using this extension we see that the
same $u$ can be recovered by minimizing the energy functional $J_a$
among all function $u$ with fixed Dirichlet condition on $(\partial
D)^+$, which is exactly the problem we stated in the
beginning of the introduction.

\section{Preliminaries}\label{sec:preliminaries}

In this section we present certain results
about the minimizers of \eqref{eq:jfunc} and $a$-harmonic functions
that we will utilize later on.

We start with a remark that the results that we are interested in are
local in nature and therefore without loss of generality we may assume
that $D=B_1$ in \eqref{eq:jfunc}. Moreover, it will also be convenient
to extend the minimizers $u$ from $B_1^+$ to the full ball $B_1$ by
even symmetry
$$
u(x',-x_n)=u(x',x_n).
$$
Then such $u$ is in $W^{1,2}(B_1,|x_n|^a)$ and is the minimizer of the functional
\begin{equation}\label{eq:jfunc-B1}
J_a(u)=\frac12\int_{B_1}|\nabla u|^2 |x_n|^a\d x+2\int_{B_1'}(\lambda_+
u^++\lambda_- u^-)\d x'
\end{equation}
subject to the boundary condition
\begin{equation}\label{eq:jfunc-2-B1}
u=g\quad\text{on }\partial B_1,
\end{equation}
in the sense of traces, where $g\in L^2(\partial B_1,|x_n|^a)$ is also evenly
extended to $\partial B_1$ from $(\partial B_1)^+$.

The existence of minimizers of
\eqref{eq:jfunc-B1}--\eqref{eq:jfunc-2-B1}, follows by the direct
method in the calculus of variation, since $J_a$ is continuous on
$W^{1,2}(B_1,|x_n|^a)$. The uniqueness follows from the strict
convexity of $J_a$ on the closed convex subset of functions in
$W^{1,2}(B_1,|x_n|^a)$ satisfying \eqref{eq:jfunc-2-B1}. The next
lemma identifies the variational inequality satisfied by the minimizers.

\begin{lem}[Variational inequality]\label{lem:normalbound} Let $u$ be a minimizer of \eqref{eq:jfunc-B1}. Then for
  $x'\in B_1'$ we have
\begin{equation}\label{eq:normallim}
-\lim_{x_n\to 0^+} x_n^a u_{x_n}(x',x_n)\in \bar
F(u)=\begin{cases}\{\lambda_+\}& u>0\\ [-\lambda_-,\lambda_+], &
  u=0\\\{-\lambda_-\}, & u<0
\end{cases}
\end{equation}
in the sense that there exists a measurable function $f$ in $B_1'$
such that 
\begin{align}\nonumber
f(x')\in \bar F (u(x',0))&\quad\text{for a.e. }x'\in B_1',\text{
  and}\\
\label{eq:normallim-weak}
-\int_{B_1^+}|x_n|^a \nabla u \nabla \psi\d x=\int_{B_1'} f\psi\d x'&\quad\text{for any }\psi\in C^\infty_0(B_1).
\end{align}
\end{lem}
\begin{rem} Taking $\psi(x',x_n)=\gamma(x')\eta(x_n)$  in
  \eqref{eq:normallim-weak}  with $\gamma\in
  C^\infty_0(B_1')$ and $\eta\in C^\infty_0([0,1))$, $\eta\equiv 1$
  near $0$,
 we obtain that the convergence in
 \eqref{eq:normallim} can be understood also in the sense of distributions
\begin{equation}
\lim_{x_n\to 0+}\int_{B_1'} x_n^a u_{x_n}(x',x_n)\gamma(x')\d
x'=\int_{B_1'}f\gamma \d x'.
\end{equation}
\end{rem}

\begin{proof} In fact, we will prove something more general: if $u$ is
  a minimizer of the functional 
$$
J(u)=\int_{B_1^+}|\nabla u|^2x_n^a+2\int_{B_1'}\Psi(u),
$$
with a fixed Dirichlet data on $(\partial B_1)^+$,
where $\Psi(s)$ is a convex function of $s\in\R$, then
$$
\int_{B_1'}
\Psi'(u-)\psi\leq -\int_{B_1^+}|x_n|^a \nabla u \nabla \psi
\leq \int_{B_1'}
\Psi'(u+)\psi
$$
for any nonnegative $\psi\in C^\infty_0(B_1)$. Indeed, for such $\psi$
consider a competing function $u+\epsilon\psi$. From the minimality of
$u$ we then have
\begin{align*}
0\leq \frac{J(u+\epsilon\psi)-J(u)}{\epsilon}&=2\int_{B_1^+} x_n^a\nabla
u\nabla \psi+\epsilon\int_{B_1^+}|\nabla \psi|^2x_n^a\\
&\qquad+\int_{B_1'}\frac{\Psi(u+\epsilon\psi)-\Psi(u)}{\epsilon}.
\end{align*}
Now, noting that $(\Psi(u+\epsilon\psi)-\Psi(u))/\epsilon$ is
monotone in $\epsilon>0$ and converges to $\Psi'(u+)\psi$, by the
monotone convergence theorem we will arrive at
$$
 -\int_{B_1^+}|x_n|^a \nabla u \nabla \psi
\leq \int_{B_1'}\Psi'(u+)\psi.
$$
The inequality from below is proved similarly.
\end{proof}

The minimizers of \eqref{eq:jfunc-B1} also enjoy the following
comparison principle. 

\begin{lem}[Comparison principle] \label{lem:comparisonprinciple}
Let $u,v$ be two minimizers of the functional \eqref{eq:jfunc-B1} with $u|_{\partial B_1} \leq v|_{\partial B_1}$. Then $u \leq v$ in $B_1$. 
\end{lem}

\begin{proof}
If we define $\overline w:= \max \{u,v\}$ and $\underline w:= \min \{u,v\}$, it is straightforward to verify that 
\[
J(\overline w) + J(\underline w) = J(u) + J(v)
\]
Since the functional is strictly convex on functions with the same
boundary data, minimizers are unique. Since $\overline w |_{\partial
  B_1} = v|_{\partial B_1}$ and $\underline w |_{\partial B_1} =
u|_{\partial B_1}$, we may therefore conclude that $u=\underline w$,
$v=\overline w$, readily implying that  $u \leq v$ in $B_1$.  
\end{proof}

\begin{cor} \label{cor:nondegeneracy}
If the boundary data are symmetric about the line $(0,\dots ,0,x_n)$, then the minimizer is symmetric about the line $(0,\dots, 0 ,x_n)$.
\end{cor}

\begin{proof} Any rotation will be a minimizer, and minimizers are unique.
\end{proof}

\begin{lem}\label{lem:upm-a-subharm} Let $u$ be a minimizer of \eqref{eq:jfunc-B1}. Then
  $u^\pm$ are $a$-subharmonic functions, i.e., $L_a
  u^\pm\geq 0$ in the weak sense
$$
\int_{B_1}\nabla (u^\pm)\nabla \psi |x_n|^a \leq 0
$$
for any nonnegative $\psi\in W^{1,2}_0(B_1,|x_n|^a)$.
\end{lem}
\begin{proof} For the nonnegative test function $\psi\in
  W^{1,2}_0(B_1,|x_n|^a)$ and small $\epsilon>0$, let 
$$
u_\epsilon=(u-\epsilon\psi)^+-u_-.
$$
Since $u=u_\epsilon$ on $\partial B_1$, from minimality of $u$ we must have $J_a(u)\leq
J_a(u_\epsilon)$. Noticing that, $u_\epsilon^+=(u-\epsilon\psi)^+\leq u^+$
and $u_\epsilon^-=u^-$ we therefore have
\begin{align*}
\int_{B_1\cap\{u>0\}}|\nabla u|^2|x_n|^a&\leq \int_{B_1\cap\{u>
  \epsilon\psi\}}|\nabla u-\epsilon\nabla \psi|^2|x_n|^a\\
&\leq \int_{B_1\cap\{u>0\}}|\nabla u-\epsilon\nabla \psi|^2|x_n|^a.
\end{align*}
This readily implies that
$$
\int_{B_1}\nabla (u^+)\nabla \psi |x_n|^a=\int_{B_1\cap\{u>0\}} \nabla
u\nabla \psi |x_n|^a\leq 0. 
$$
This proves $a$-subharmonicity of $u^+$. Arguing similarly, we
establish the same fact also for $u^-$. 
\end{proof}

\begin{cor}[Energy inequality]\label{cor:energy} Let $u$ be a minimizer of \eqref{eq:jfunc-B1}. Then we
  have the following inequality
\[
\int_{B_r}|\nabla u|^2 |x_n|^a \leq \frac{C}{r^2}\int_{B_{2r}} u^2|x_n|^a.
\]
\end{cor}
\begin{proof} Since $u^\pm$ are nonnegative and $a$-subharmonic, the
  standard proof of the Caccioppoli inequality (with test functions $u^\pm\eta^2$) applies and gives
$$
\int_{B_r}|\nabla u_\pm|^2|x_n|^a\leq \frac{C_n}{r^2}\int_{B_{2r}} u_\pm^2|x_n|^a.
$$
Combining these two inequalities, we complete the proof. 
\end{proof}

\begin{cor}[Local boundedness]\label{cor:loc-bdd} Let $u$ be a minimizer
  \eqref{eq:jfunc-B1}. Then 
$$
\| u\|_{L^\infty(B_{1/2})}\leq C(n,a)\|u\|_{L^2(B_1,|x_n|^a)}
$$
\end{cor}
\begin{proof} This follows from $L^\infty$-$L^2$ estimates for
  $a$-subharmonic functions $u^\pm$, see \cite{FKS82}.
\end{proof}


For our further study, we will need some properties of $a$-harmonic
functions. We start with the observation that Corollaries~\ref{cor:energy} and
\ref{cor:loc-bdd} are applicable also to $a$-harmonic functions (e.g.,
take $\lambda_\pm=0$ in the proofs above). Besides, since $|x_n|^a$ is
an $A_2$ weight, we have the following result from \cite{FKS82}.

\begin{lem}[H\"older continuity]\label{lem:a-harm-holder}
\pushQED{\qed}
 Let $u$ be $a$-harmonic in $B_1$. Then for some
  $\alpha>0$, 
$u\in C^{0,\alpha}(B_{1/2})$ and
\[
\|u\|_{C^{0,\alpha}(B_{1/2})}\leq C(n,a,\alpha)\|u\|_{L^2(B_1)}.\qedhere
\]
\popQED
\end{lem}
Next, we the following derivative estimates in $x'$ directions
(see \cite[Corollary~2.5]{CSS08}).

\begin{lem}[Interior estimates]\label{lem:thin-grad-est}
\pushQED{\qed}
Let $u$ be
  $a$-harmonic in a ball $B_r(x)$. Then for any positive integer $k$
\[
\sup_{B_{r/2}(x)}|D_{x'}^k u|\leq\frac{C(n,a,k)}{r^k}\osc_{B_r(x)}
u\qedhere
\]
\popQED
\end{lem}

Note that because of the weight $|x_n|^a$, similar estimates for the
derivatives in $x_n$ must be weighted accordingly.

Further, at several points we will need the following result \cite[Lemma~2.7]{CSS08}.

\begin{lem}[Liouville-type theorem]\label{lem:liouville} Let $u$ be
  $a$-harmonic in $\R^n$, even with respect to $x_n$ and 
$$
|u(x)|\leq C(1+|x|^k), 
$$
for some $C,k>0$. Then $u$ is a polynomial.\qed
\end{lem}

We next state an Almgren-type frequency formula, which first appeared
in \cite[Theorem~6.1]{CS07}:
\begin{lem}[Almgren-type frequency formula] \label{lem:almgren}Let $u$ be $a$-harmonic in $B_1$. Then the function
$$
N(r)=N(u,r)=\frac{\displaystyle r\int_{B_r}|\nabla u|^2 |x_n|^a}{\displaystyle\int_{\dd B_r}u^2 |x_n|^a }
$$
is increasing in $r$, for $r\in (0,1)$. Moreover, $N(r)$ is constant in $(r_1,r_2)$ if and only if $u$ is homogeneous of degree $N(r_1)$ in $B_{r_2}\setminus B_{r_1}$.\qed
\end{lem}

Following the proof of \cite[Lemma~4.1]{Wei01} and using Lemma~\ref{lem:liouville} and Lemma~\ref{lem:almgren} we have the corollary below: 
\begin{cor} \label{cor:weisscor} Let $u$ be $a$-harmonic in $B_1$,
  even in $x_n$,  with $u(0)=0$. If $a<0$ assume also that $|\nabla_{x'} u(0)|=0$. Then
$$
\int_{B_r}|\nabla u|^2|x_n|^a\geq C(a)\int_{\dd B_r}u^2 |x_n|^a ,
$$
where $C(a)=1$ if $a\geq 0$ and $C(a)=2$ if $a<0$. Moreover, equality holds if and only if $u$ is homogeneous of degree $C(a)$.
\end{cor}
\begin{proof} We first prove the inequality by contradiction. If the statement is false then there is $r_0\in (0,1)$ such that $N(r_0)<C(a)$. By Lemma \ref{lem:almgren}, this implies
$N(r)<C(a)$ for all $r\in (0,r_0)$. Let 
$$
w_r(x)=\frac{w(rx)}{\displaystyle \Big(\int_{\dd B_1} w^2(rx)
  |x_n|^a\Big)^\frac12}.
$$
Then
\begin{equation}\label{eq:int1}
\int_{\dd B_1}w_r^2|x_n|^a = 1,
\end{equation}
$w_r$ is bounded in $W^{1,2}(B_1,|x_n|^a)$ and $a$-harmonic in
$B_1$. By the compactness of the trace operator,
Lemma~\ref{lem:a-harm-holder} and Corollary~\ref{cor:energy}, we can find a sequence $r_j\to 0+$ such that $w_{r_j}$ converges strongly in
$L^2(\partial B_1,|x_n|^a)\cap W^{1,2}(B_{1/2},|x_n|^a)\cap C(B_{1/2})$
(and also in $C^1_{x'}(B_{1/2})$ if $a<0$ from Lemma \ref{lem:thin-grad-est}) to a limit function $w_0$. Hence, for $s\in
(0,1/2)$ there holds 
$$
N(0+,w)=\lim_{r_j\to 0} N(sr_j,w)=\lim_{r_j\to 0} N(s,w_{r_j})=N(s,w_0),
$$
so that $N(s,w_0)$ is constant. Therefore, by Lemma
\ref{lem:almgren}, $w_0$ is homogeneous of degree $N(0+,w)<C(a)$ in
$B_{1/2}$. We can extend $w_0$ to $a$-homogeneous function in whole
$\R^n$. Note that $a$-harmonic functions are real analytic off the
thin space $\R^{n-1}\times\{0\}$, and therefore $a$-harmonic continuations are unique.

Now, if $a\geq 0$, Lemma \ref{lem:liouville} implies that $w_0=0$
contradicting \eqref{eq:int1}. 

If $a<0$, then Lemma \ref{lem:liouville} implies that $w_0$ is linear. Due to the extra assumption when $a<0$, $w_0=0$ and a contradiction is reached again.

Finally, if equality holds for some $r_0\in (0,1)$, then it must hold for all $r<r_0$, so by Lemma \ref{lem:almgren} and unique continuation, $u$ must be homogeneous of degree $C(a)$.
\end{proof}
The next result establishes a Weiss-type monotonicity formula, which
will be an important tool in our study. Since the proof is rather
technical, for readers' convenience we have moved it to the appendix.

\begin{thm}[Weiss-type monotonicity formula]  \label{thm:weiss} 
Let $u$ be a minimizer of \eqref{eq:jfunc-B1}. Then the functional 
\begin{align*}
r\mapsto W(r)=W(r,u)&:=\frac1{r^{n-a}}\int_{B_r}|\nabla u|^2 |x_n|^a
+\frac4{r^{n-a}}\int_{B_r'}(\lambda_+ u^++\lambda_-u^-)\\
&\qquad - \frac{1-a}{r^{n+1-a}}\int_{\dd B_r}u^2 |x_n|^a 
\end{align*}
is nondecreasing for $0<r<1$. Moreover, $W$ is constant for $r\in (r_1,r_2)$ if and only if $u$ is homogeneous of degree $1-a$ in the ring $r_1<|x|<r_2$.
\end{thm}

As we will see, the case $a=0$ will require a special treatment. The
following result, known as the Alt-Caffarelli-Friedman (ACF)
monotonicity formula, will be instrumental in the study of that case.

\begin{lem}[ACF monotonicity formula]   \label{lem:ACF}
Let $\{w_+, w_-\}$ be a pair of nonnegative continuous subharmonic functions in $B_R$ such that 
$ w_+ \cdot w_- =0$ in $B_R$. Then the functional
\[
r \mapsto \Phi(r,w_+,w_-):= \frac{1}{r^4} \int_{B_r}{\frac{|\nabla w_+|^2}{|x|^{n-2}}}
 \int_{B_r}{\frac{|\nabla w_-|^2}{|x|^{n-2}}}
\]
is finite and nondecreasing for $0<r<R$.
\end{lem}
For the proof of Lemma \ref{lem:ACF} we refer to \cite{bCS05,PSU12}, or
the original paper \cite{ACF84}. We also note that we can change the
center of the ball to $x_0$ by replacing $\Phi$ with $\Phi^{x_0}$
given by
\[
 \Phi^{x_0}(r,w_+,w_-):= \frac{1}{r^4} \int_{B_r(x_0)}{\frac{|\nabla w_+|^2}{|x-x_0|^{n-2}}}
 \int_{B_r(x_0)}{\frac{|\nabla w_-|^2}{|x-x_0|^{n-2}}}
\]

\begin{rem}  \label{rem:ACF} 
For $w_r(x): = w(rx)/r$, the functional $\Phi (r)$ enjoys the following rescaling property:
$$
\Phi(r,w_+,w_-) = \Phi(1,(w_+)_r, (w_-)_r)
$$
\end{rem}
\begin{rem} If $u$ is a minimizer of \eqref{eq:jfunc-B1} and $a=0$, then $w_\pm=u^\pm$ are subharmonic so that Lemma \ref{lem:ACF} applies to the pair.
\end{rem}

\section{Regularity of minimizers}\label{sec:regul-minim}

 In this section we prove the
H\"older continuity of the minimizers of \eqref{eq:jfunc-B1} for all
$a\in (-1,1)$ and $C^{1,\alpha}$ regularity along the thin space when $a\in (-1,0)$. As we
will see these regularity results are not optimal, but necessary for
technical reasons. The optimal regularity is established in the next section.

\begin{thm}[H\"older continuity]\label{thm:holder} Let $u$ be a minimizer of
  \eqref{eq:jfunc-B1}. Then $u\in C^{0,s}(B_1)$,
  $s=(1-a)/2$. Moreover, there exists $C=C(a,n,\lambda_\pm)>0$ such that 
$$
\|u\|_{C^{0,s}(B_{1/2})}\leq C\|u\|_{L^2(B_1, |x_n|^a)}.
$$
\end{thm}
\begin{proof} The proof is almost identical with the proof of Theorem
  3.1 in \cite{a11}, we spell out the details below.

By Corollary~\ref{cor:loc-bdd}, we know that $u$ is bounded in
$B_{3/4}$. Then put
$$
C_0=4\max\{\lambda_+,\lambda_-\}\sup_{B_{3/4}}|u|.
$$
For $r\in (0,3/4)$, let $v$ be the function that is $a$-harmonic in
$B_r$ and equals $u$ on $\dd B_r$ (in other words, $v$ is the
$a$-harmonic replacement of $u$ on $B_r$). Then since $u$ minimizes $J_a$ on all balls inside $B_1$, there holds
\begin{equation}\label{eq:uv}
\int_{B_r}\left(|\nabla u|^2-|\nabla v|^2\right)|x_n|^a\leq 4\int_{B_r'}(\lambda_+v^++\lambda_-v^-)\leq C_0r^{n-1},
\end{equation}
due to the estimate
$$
\sup_{B_r} |v|\leq \sup_{B_{r}}|u|.
$$
Since $u=v$ on $\partial B_r$ we also have
$$
\int_{B_r}\nabla v\cdot (\nabla u-\nabla v) |x_n|^a
=0,
$$
which implies
$$
\int_{B_r}\left(|\nabla u|^2-|\nabla v|^2 \right)|x_n|^a=\int_{B_r}|\nabla u-\nabla v|^2 |x_n|^a,
$$
so that \eqref{eq:uv} becomes
$$
\int_{B_r}|\nabla u-\nabla v|^2 |x_n|^a\leq C_0r^{n-1}.
$$
This can now be iterated in a standard way (cf.\ Theorem 3.1 in \cite{a11}), to imply
$$
\int_{B_r}|\nabla u|^2 |x_n|^a \leq Cr^{n-1},
$$
for $0<r<3/4$, where $C$ might also depend on the
$W^{1,2}(B_{3/4},|x_n|^a)$-norm of $u$, which by the energy inequality
(see Corollary~\ref{cor:energy}) can be estimated
in terms of $L^2(B_1,|x_n|^a)$-norm of $u$. By H\"older's inequality this gives
$$
\int_{B_r}|\nabla u|\d x\leq \| |\nabla u|x_n^{a/2}\|_{L^2(B_r)}\|x_n^{-a/2}\|_{L^2(B_r)}\leq Cr^{n-(1+a)/2}.
$$
The estimate above is valid for all balls $B_r(x',0)$ centered at the
thin space, as long as they are contained in $B_{3/4}$. From arguments used in the proof of Morrey's embedding this yields the estimates
$$
\Big|u(x',0)-\dashint_{B_r(x',0)}u\Big|\leq Cr^s,\quad |u(x',0)-u(y',0)|\leq Cr^s,
$$
where $s=(1-a)/2$. Then, arguing as at the end of the proof of Theorem
3.1 in \cite{a11} we can conclude that $u$ is $s$-H\"older continuous in
$B_{1/2}$ with the required estimate on the $C^{0,s}$ norm.
\end{proof}

\begin{rem}\label{rem:holder-EL} The fact that $u$ is a minimizer of $J_a$ in $B_1$ can be
  changed with the following conditions:
\begin{enumerate}
\item $L_a u=0$ in $B_1^+$;
\item $|\lim_{x_n\to 0+} x_n^a u_{x_n}(x',x_n)|\leq \mu$ for $x'\in
  B_1'$ in the sense that
\[
\Big|\int_{B_1^+}(\nabla u\cdot\nabla \psi) x_n^a\Big|\leq
\mu\int_{B_1'}|\psi|,
\]
for any $\psi\in C^\infty_0(B_1)$ (cf.\ Lemma~\ref{lem:normalbound});
\item $|u|\leq M$ in $B_1^+$.
\end{enumerate}
Then we will have that
$$
\|u\|_{C^{0,s}(B_{1/2})}\leq C(M,\mu,n,a).
$$
Instead of \eqref{eq:uv} one will have to use
\begin{align*}
\Big|\int_{B_r}\nabla u(\nabla v-\nabla u)|x_n|^a\Big|&\leq 2\mu \Big|\int_{B_1'}|u-v|\Big|\\
&\leq C(M,\mu,n)r^{n-1}
\end{align*}
together with
$$
\int_{B_r}\nabla v(\nabla v-\nabla u)|x_n|^a=0.
$$
\end{rem}

We now turn to the $C^{1,\alpha}$ regularity along the thin space when
$a\in (0,1)$.

\begin{thm}[$C^{1,\alpha}$ regularity, $a<0$]   \label{thm:thinreg}
Let $u$ be  minimizer of \eqref{eq:jfunc-B1} in $B_1$ with $|u|\leq M$ in $B_1$. If $a\in (-1,0)$, then $u \in C^{1,\alpha}(B_{1}')$
for $\alpha< -a$. Moreover, 
\[
\| u \|_{C^{1, \alpha}(B_{1/2}')} \leq C(M,n, a,\alpha)(\|u\|_{L^2(B_1,|x_n|^a)}+\max\{\lambda_+,\lambda_-\})
\] 
\end{thm}

\begin{proof}
Let $f$ be defined as in Lemma \ref{lem:normalbound}. We extend $f$ to all of $\R^{n-1}$ by defining $f(y')=0$ if $y' \notin B_1$. 
Then $f \in L^{\infty}(\R^{n-1})$ with compact support, so we can use
the Riesz potential 
$$
I_{2s}f(x')=c_{n,s}\int_{\R^{n-1}}\frac{f(y')\d y'}{|x'-y'|^{n-2s}}
$$
as the inverse to the fractional Laplacian in the sense that $(-\Delta_{x'})^s I_{2s}f=f$. Since $a<0$, by
Proposition~2.9 in \cite{Sil07},
$I_{2s} f \in C^{1,\alpha}(\R^{n-1})$ for any $\alpha<-a$. We now use the extension
operator as in \cite{CS07} on $I_{2s} f$ to obtain an $a$-harmonic
function $v$ in the upper half plane with Dirichlet data $I_{2s} f$,
so  
\[
v|_{\R^{n-1}} = I_{2s} f.
\]
 We may evenly reflect $v$ across the thin space. If $w= u-v$, then for all $\psi \in C_0^{2}(B_1)$
\[
\int_{B_{1}'}{|x_n|^a \nabla w \nabla \psi} = 0.
\]
Then $L_a w =0$ in all of $B_1$. In particular, $w$ is  $C^{\infty}$
regular on the thin ball $B_1'$ in any $x'$ direction (cf. Lemma~\ref{lem:thin-grad-est}). Since $u=v+w$, it follows that 
\[
u \in C^{1, \alpha}(B_1')\quad\text{for } \alpha < -a.
\] 
The interior estimate on $u$ follows by the interior estimate on $v$
(see Proposition~2.9 in \cite{Sil07}) and for the $a$-harmonic
function $w$ (see Lemma~\ref{lem:thin-grad-est}). 
\end{proof}
\begin{rem}\label{rem:C1alpha-EL} Similarly to
  Remark~\ref{rem:holder-EL}, in Theorem~\ref{thm:thinreg} we may
  change the requirement that $u$ is a minimizer of $J_a$ in $B_1$
  with conditions (1)--(3) in Remark~\ref{rem:holder-EL}. Then, the
  conclusion will be that
$$
\|u\|_{C^{1,\alpha}(B_{1/2}')}\leq C(M,\mu,a,\alpha,n)
$$
for any $0<\alpha<-a$.

\end{rem}
\section{Optimal growth and regularity}\label{sec:optim-growth-regul}

Now we have all the tools needed to prove
the optimal growth of order $1-a$, which follows below.

\begin{thm}[Optimal growth]\label{thm:main1-growth} Let $u$ be a
  bounded minimizer of \eqref{eq:jfunc-B1}, with $|u|\leq M$ in $B_1$. Then the
  following holds.
\begin{enumerate}[\textup{(}i\/\textup{)}]
\item If $a\geq 0$, then 
$$
\sup_{B_r(x',0)}|u(y)-u(x',0)|\leq Cr^{1-a}
$$
for $0< r<1/2$ and $x'\in B_{1/2}'$, with $C=C(a,M,\lambda_\pm,n)$.

\smallskip
\item If $a<0$, then 
$$
\sup_{B_r(x',0)}|u(y)-u(x',0)-(y-(x',0))\cdot \nabla_{x'}u(x',0)|\leq
Cr^{1-a},
$$
for  $0<r<1/2$ and $x'\in B_{1/2}'$, with $C=C(a,M,\lambda_\pm,n)$. 
\end{enumerate}
\end{thm}

\begin{proof}
\smallskip\noindent {\bf Case i: $a\geq 0$.} For $0<r<1/2$ and $x'\in B_{1/2}'$ let
$$
S_r(u,x') = \sup_{B_r}|u-u(x',0)|.
$$ 
If the assertion does not hold then we can find a sequence of
minimizers $u_j$, points $x_j'\in B_{1/2}'$ and radii $r_j\to 0$ such that
$$
\frac{S_{r_j}(u_j,x_j')}{r_j^{1-a}}=C_j\to \infty, 
$$
and
$$
S_{r_j2^k}(u_j,x_j')\leq 2^{k(1-a)}S_{r_j}(u_j,x_j'),
$$
for all nonnegative integers $k$ such that $2^kr_j\leq 1$. Let now
$$
v_j(x)=\frac{u_j(r_jx+(x'_j,0))-u_j(x'_j,0)}{C_j {r_j}^{1-a}}.
$$
Then 
\begin{enumerate}
\item $v_j$ is even in $x_n$-variable
\item $v_j$ satisfies
$L_a v_j=0$ in $B_R^+$
for $R<1/r_j$;

\item For any $\psi\in C^\infty_0(B_R)$
$$
\int_{B_R^+} (\nabla v_j\cdot\nabla \psi)x_n^a\leq \frac{\max\{\lambda_+,\lambda_-\}}{C_j}\int_{B_R'}|\psi|;
$$
\item $ \sup_{B_{2^k}}|v_j|\leq 2^{k(1-a)}$, whenever $2^kr_j\leq 1$;
\item $v_j(0)=0$; 
\item $\sup_{B_1} |v_j|=1$.
\end{enumerate}
By Theorem~\ref{thm:holder} and Remark~\ref{rem:holder-EL}, $v_j$ is uniformly bounded in
$C^{0,s}(B_R)$ for all $R<1/r_j$.  Passing to the limit, we obtain
that $v_j$ converges (over a subsequence) to a function $v_0\in
C^{0,s}(\R^n)$ which satisfies
$$
\int_{\R^n}(\nabla v_0\cdot\nabla \psi)|x_n|^a=0
$$
for any $\psi\in C^\infty_0(\R^n)$. In other words, $v_0$ is
$a$-harmonic in $\R^n$, even in $x_n$-variable. Besides, we will have that
$$
|v_0(x)|\leq C(1+|x|^{1-a}),\quad v_0(0)=0,\quad \sup_{B_1} |v_0|=1.
$$
The Liouville-type theorem in Lemma \ref{lem:liouville} now implies implies that $v_0$ is a
polynomial. 

At this point we need to split the proof into two different cases depending on $a$. 

\smallskip
\noindent {\bf Subcase i.a: $a>0$.} Since the growth at infinity is
less than linear $v_0$ must be constant and thus $v_0\equiv 0$, which
is clearly a contradiction with the properties of $v_0$ above.

\smallskip
\noindent {\bf Subcase i.b: $a=0$.} Again, due to the growth condition,
we can deduce that $v_0$ is linear. To get a contradiction, we intend
to use the ACF monotonicity formula, but in order to do so, we need to make
the following observation. For the sequences $u_j$, $x_j'$ and $r_j$
in our contradictory assumption, without loss of generality we may
assume that the following holds:
\begin{itemize}
\item[($\alpha$)] either $u_j(x_j',0)=0$ for all $j$, 
\item[($\beta$)] $B_{r_j}'(x_j')\subset\{u_j>0\}$ for all $j$, or 
\item[($\gamma$)] $B_{r_j}'(x_j')\subset\{u_j<0\}$ for all $j$.
\end{itemize}
Indeed, in the case when $x_j'\not \in\Gamma_{u_j}$ and
$r_j<\dist(x_j',\Gamma_{u_j})$, we may simply replace $x_j'$ with its
closest point $y_j'\in \Gamma_{u_j}$ and $r_j$ with $2r_j$ and convert
it to case ($\alpha$).

Cases ($\alpha$)--($\gamma$) translate into the following ones for functions $v_j$.
\begin{itemize}
\item[($\alpha'$)] $v_j(0,0)=0$ for all $j$,
\item[($\beta'$)] $v_j>0$ in $B_1'$ for all $j$,
\item[($\gamma'$)] $v_j<0$ in $B_1'$ for all $j$.
\end{itemize}

In the cases ($\beta'$) and ($\gamma'$) we obtain immediately that the limit
$v_0$ has a sign in $B_1'$. Since we also know that $v_0$ is linear
and even in $x_n$, this immediately implies that $v_0$ must be
identically zero, a contradiction with property (6) of $v_j$. In the
case ($\alpha'$), which is the same as ($\alpha$), we observe that we have the
additional property that 
$$
\Delta (u_j^\pm)\geq 0\quad\text{in } B_1,
$$
see Lemma~\ref{lem:upm-a-subharm}.
Then we can apply the ACF monotonicity formula
(Lemma~\ref{lem:ACF}). We will have that for any $s>0$ and large $j$
\begin{align*}
\Phi(s,v_j^+,v_j^-)&=\frac{1}{C_j^2}\Phi^{x_j}(sr_j,u_j^+,u_j^-)\\
&\leq
\frac{1}{C_j^2}\Phi^{x_j}(1/4,u_j^+,u_j^-)\leq
\frac{C(n)}{C_j^2}\|u_j\|_{L^2(B_1)}^2\to 0.
\end{align*}
Hence, $\Phi(s,v_0^+,v_0^-)=0$ for $s>0$, and we can conclude that
either $v_0^+=0$ or $v_0^-=0$. Since $v_0$ is linear and even in
$x_n$, we again conclude that $v_0$ is identically zero, which
contradicts property (6) of $v_j$.

\smallskip
\noindent {\bf Case ii: $a<0$.} Here we define instead
$$
S_r(u,x') = \sup_{y\in B_r(x',0)}|u(y)-u(x',0)-y'\cdot \nabla_{x'}u(x',0)|.
$$ 
If the assertion does not hold then we can find a sequence of
minimizers $u_j$, points $x_j'\in B_{1/2}'$ and a radii $r_j\to 0$ such that
$$
\frac{S_{r_j}(u_j,x_j')}{r_j^{1-a}}=C_j\to \infty, 
$$
and
$$
S_{r_j2^k}(u_j,x_j')\leq 2^{k(1-a)}S_{r_j}(u_j, x_j'),
$$
for all nonnegative integers $k$ such that $2^kr_j\leq 1$. Let now
$$
v_j(x)=\frac{u_j(r_jx+(x',0))-u_j(x',0)-(r_j x')\cdot \nabla_{x'}u_j(x_j',0)}{C_j{r_j}^{1-a}}.
$$
Then, again as before,
\begin{enumerate}
\item  $v_j$ is even in $x_n$-variable
\item $v_j$ satisfies
$L_a v_j=0$ in $B_R^+$
for $R<1/r_j$;

\item For any $\psi\in C^\infty_0(B_R)$
$$
\int_{B_R^+} (\nabla v_j\cdot\nabla \psi)x_n^a\leq \frac{\max\{\lambda_+,\lambda_-\}}{C_j}\int_{B_R'}|\psi|;
$$
\item $ \sup_{B_{2^k}}|v_j|\leq 2^{k(1-a)}$, whenever $2^kr_j\leq 1$;
\item $v_j(0)=|\nabla_{x'}v_j(0)|=0$; 
\item $\sup_{B_1} |v_j|=1$.
\end{enumerate}
By Theorem~\ref{thm:holder}, Remark~\ref{rem:holder-EL}, Theorem~\ref{thm:thinreg} and Remark~\ref{rem:C1alpha-EL}, there is a subsequence, again labelled
$v_j$, converging in $C^{0,s}_\textup{loc}\cap C^{1,\alpha}_\textup{loc}(\R^{n-1}\times 0)$ and then by (3) above also weakly in $W^{1,2}_\textup{loc}(|x_n|^a)$ to a limit function
$v_0$. By passing to the limit in (3) and (4) above, we see that

$$
L_a v_0=0\quad\text{in }\R^n
$$
and
$$
|v_0(x)|\leq C(1+|x|^{1-a}),\quad v_0(0)=|\nabla_{x'}v_0(0)|=0,\quad
\sup_{B_1}|v_0|=1.
$$
Since $v_0$ is also even in $x_n$ variable, Lemma \ref{lem:liouville} now implies that $v_0$ is a
polynomial. Since the growth at infinity is less than quadratic, $v_0$
must be linear. The properties of $v_0$ above imply that $v_0\equiv 0$, a contradiction.
\end{proof}

\begin{rem}\label{rem:optgrowth-EL} As in Remark \ref{rem:holder-EL} and \ref{rem:C1alpha-EL}, the assertions of Theorem \ref{thm:main1-growth} hold true when the requirement that $u$ is a minimizer of $J_a$ in $B_1$ is replaced by conditions (1)--(2) in Remark \ref{rem:holder-EL}, for $a\neq 0$. The only change in the conclusion would be that now $C=C(a,M,\mu,n)$. The result is not true for $a=0$ as the example $u=\Re z\ln z$ shows. The lacking information is the fact that for minimizers, $u^\pm$ are subsolutions.
\end{rem}

With the proposition above, we can now prove the optimal
regularity. More precisely, we prove the following version of Theorem
\ref{thm:main-opt-reg}. 

\begin{thm}[Optimal regularity]\label{thm:main1-expanded} Let $u$ be a
  bounded minimizer of \eqref{eq:jfunc-B1}, with $0\in
  \Gamma^+\cup \Gamma^-$ and $|u|\leq M$ in $B_1$. Then the
  following holds.
\begin{enumerate}[\textup{(}i\/\textup{)}]
\item If $a\geq 0$, then $u\in C^{0,1-a}(B_{1/4})$ with 
$$
\|u\|_{C^{0,1-a}(B_{1/4})}\leq C(a,M,\lambda_\pm,n).
$$
\smallskip
\item If $a<0$, then $u\in C^{1,-a}(B_{1/4}^\pm\cup B_{1/4}')$ with  
$$
\|u\|_{C^{1,-a}(B_{1/4}^\pm\cup B_{1/4}')}\leq C(a,M,\lambda_\pm,n).
$$
\end{enumerate}
\end{thm}

\begin{proof} \textbf{Case i: $a\geq 0$.} First, form the growth estimate in
Theorem~\ref{thm:main1-growth} and
Lemma~\ref{lem:thin-grad-est} applied in the ball $B_{x_n}(x',x_n)$,
we obtain that
\begin{equation}\label{eq:ux'-est}
|\nabla_{x'}u(x)|\leq C x_n^{-a},\quad\text{for }x\in B_{1/4}^+,
\end{equation}
where $C=C(a,M,\lambda_\pm, n)$. Besides, we claim that a similar
estimate holds also in $x_n$-direction:
\begin{equation}\label{eq:uxn-est}
|u_{x_n}(x)|\leq C x_n^{-a},\quad\text{for }x\in B_{1/4}^+.
\end{equation}
To show this estimate, we note that from Lemma \ref{lem:normalbound}, we know that the function $\tilde u=x_n^a
u_{x_n}$ is uniformly bounded in $B_{1/2}'$, with the bound depending
on $\lambda^\pm$. Moreover, it is known that $L_{-a}\tilde u=0$ in
$B_{1}^+$ (see for instance \cite{CS07} for an explanation). Next, let $v$ be the bounded
solution of
$$
L_{-a} v=0\quad\text{in }\R^n_+,\qquad v(x',0)=\tilde u(x',0)\zeta(x'),\quad x'\in\R^{n-1},
$$
with some cutoff function $\zeta\in C^\infty_0(B_1')$ which equals $1$
on $B_{1/2}'$. Then, by using the Poisson formula (see \cite{CS07}),
we will have
$$
\|v\|_{L^\infty(B_{1/2}^+)}\leq C\|\tilde u\zeta \|_{L^\infty(B_{1/2}')}\leq C(a,M,\lambda^\pm,n).
$$
Furthermore, let $w=u-v$ and extend $w$ in an odd manner across the
thin space. Then, $w$ is $(-a)$-harmonic in $B_{1/2}$. Hence, from
Lemma~\ref{lem:upm-a-subharm} and Corollaries~\ref{cor:energy} and
\ref{cor:loc-bdd}, we have
\begin{align*}
\|w\|_{L^\infty(B_{1/2}^+)}&\leq C\|w\|_{L^2(B_{3/4},|x_n|^{-a})}\\
&= C\|\nabla u\|_{L^2(B_{3/4},|x_n|^a)}\leq C\sup_{B_1} |u|=C M.
\end{align*}
Combining the estimates for $v$ and $w$, we obtain a similar estimate
for $\tilde u$, which implies \eqref{eq:uxn-est}. 

With estimates \eqref{eq:ux'-est}--\eqref{eq:uxn-est}, we can now
prove the H\"older continuity of $u$. Let $x, y\in B_{1/4}^+\cup B_{1/4}'$ and
without loss of generality assume that $x_n\leq y_n$. Consider the
following possibilities:
\begin{enumerate}
\item $|x-y|\geq y_n/2$. In this case, by
  Theorem~\ref{thm:main1-growth},  we have
\begin{align*}
|u(x)-u(y)|&\leq |u(x',0)-u(y',0)|+|u(x)-u(x',0)|+|u(y)-u(y',0)|\\
&\leq
C|x'-y'|^{1-a}+Cx_n^{1-a}+C y_n^{1-a}\leq C|x-y|^{1-a}.
\end{align*}

\item $|x-y|< y_n/2$. In this case, $x,y\in B_{y_n/2}(y)$ and we have
  that by \eqref{eq:ux'-est}--\eqref{eq:uxn-est}
$$
|\nabla u|\leq Cy_n^{-a}\quad\text{in } B_{y_n/2}(y). 
$$
Hence
$$
|u(x)-u(y)|\leq C y_n^{-a}|x-y|\leq C|x-y|^{1-a}.
$$
\end{enumerate}
Thus, we obtain the desired H\"older continuity of $u$ in
$B_{1/4}^+\cup B_{1/4}'$. By symmetry in $x_n$ variable, we therefore
have the same H\"older continuity in $B_{1/4}$.

\medskip\noindent
\textbf{Case ii: $a<0$.} First, note that the estimate \eqref{eq:uxn-est} still holds in this
case. To prove the desired $C^{1,-a}$ regularity we will obtain similar estimates for the second derivatives.

To this end, take a point $x=(x',x_n)\in B_{1/4}^+$. Then $\tilde u=x_n^a u_{x_n}$ is $(-a)$-harmonic in
$B_{x_n}(x)$ and applying Lemma~\ref{lem:thin-grad-est} we obtain
\begin{equation}\label{eq:ux'xn-est}
|\nabla_{x'} u_{x_n}(x)|\leq C x_n^{-1-a}\quad\text{for }x\in B_{1/4}^+.
\end{equation}
Further, using that
$w(y)=u(y)-u(x',0)-(y-(x',0))\cdot\nabla_{x'} u(x',0) $, is
$a$-harmonic in $B_{x_n}(x)$, we obtain by
Lemma~\ref{lem:thin-grad-est} 
\begin{equation}\label{eq:ux'x'-est}
|D^2_{x'}u(x)|=|D^2_{x'} w(x)|\leq \frac{C}{x_n^2} \sup_{B_{{x_n}/{2}}(x)} |w|\leq C{x_n^{-1-a}},\quad\text{for }x\in B_{1/4}^+,
\end{equation}
where the last inequality follows from the optimal growth estimate in
Theorem~\ref{thm:main1-growth}. Now, using the equation for $u$, we conclude that 
\begin{equation}\label{eq:uxnxn-est}
|u_{x_nx_n}(x)|=|-\lap_{x'} u(x)-a x_n^{-1}u_{x_n}(x)|\leq
C{x_n^{-1-a}},\quad\text{for }x\in B_{1/4}^+.
\end{equation}
Combining \eqref{eq:ux'xn-est}--\eqref{eq:uxnxn-est}, we arrive at
\begin{equation}
  \label{eq:D2u-est}
  |D^2u(x)|\leq Cx_n^{-1-a},\quad\text{for }x\in B_{1/4}^+.
\end{equation}
Now taking $x,y\in B_{1/4}^+\cup B_{1/4}'$ and repeating the the
arguments as at the end of Case i, we readily conclude that
$$
|\nabla u(x)-\nabla u(y)|\leq C |x-y|^{-a}.
$$
(In the estimate of $u_{x_n}$ we use that $u_{x_n}$ is identically
zero on $B_{1/2}'$, by estimate \eqref{eq:uxn-est}.) Furthermore,
using Theorem~\ref{thm:thinreg}, integrating \eqref{eq:ux'xn-est} from
$(x',0)$ to $(x',x_n)$, and using \eqref{eq:uxn-est} one more time, we
also establish that
$$
|\nabla u|\leq C\quad\text{for }x\in B_{1/4}^+.
$$
Hence,
$$
\|u\|_{C^{1,-a}(B_{1/4}^+\cup B_{1/4}')}\leq C(a,M,\lambda^\pm,n), 
$$
as required.
\end{proof}

The proof of Theorem~\ref{thm:main-opt-reg} is now immediate.

\begin{proof}[Proof of Theorem~\ref{thm:main-opt-reg}] The proof
  follows from the local boundedness (see 
  Corollary~\ref{cor:loc-bdd}), Theorem~\ref{thm:main1-expanded}, and
  a simple covering argument by balls. 
\end{proof}

From Remarks \ref{rem:holder-EL}, \ref{rem:C1alpha-EL} and \ref{rem:optgrowth-EL}, we also have the following version of Theorem \ref{thm:main1-expanded}, when $a\neq 0$.

\begin{thm}\label{thm:main1-EL} Let $a\neq 0$. Assume
\begin{enumerate}
\item $L_a u=0$ in $B_1^+$;
\item $|\lim_{x_n\to 0+} x_n^a u_{x_n}(x',x_n)|\leq \mu$ for $x'\in
  B_1'$ in the sense that
\[
\Big|\int_{B_1^+}(\nabla u\cdot\nabla \psi) x_n^a\Big|\leq
\mu\int_{B_1'}|\psi|,
\]
for any $\psi\in C^\infty_0(B_1)$;
\item $|u|\leq M$ in $B_1^+$.
\end{enumerate}
Then the
  following holds:
\begin{enumerate}[\textup{(}i\/\textup{)}]
\item If $a> 0$, then $u\in C^{0,1-a}(B^+_{1/4})$ with 
$$
\|u\|_{C^{0,1-a}(B^+_{1/4})}\leq C(a,M,\mu,n).
$$
\smallskip
\item If $a<0$, then $u\in C^{1,-a}(B_{1/4}^+\cup B_{1/4}')$ with  
$$
\|u\|_{C^{1,-a}(B_{1/4}^+\cup B_{1/4}')}\leq C(a,M,\mu,n).
$$
\end{enumerate}
\end{thm}

\section{Nondegeneracy}\label{sec:nondegeneracy}
 In this section we prove a certain nondegeneracy property for the minimizers that will be instrumental for the proof of our
second main result, Theorem~\ref{thm:main-free-bound}. We follow the outline as given in \cite{AP11}.

\begin{thm}[Nondegeneracy] \label{thm:nondegeneracy}
Fix $0<t<1$, and let $u$ be a minimizer of \eqref{eq:jfunc-B1}. There
exists $\epsilon > 0$ with $\epsilon$ depending only on $\lambda_\pm$
and $t$, such that if $u|_{\partial B_r} \leq  \epsilon r^{1-a} $ $(u|_{\partial B_r} \geq  -\epsilon r^{1-a})$ then 
\[
u(x) \leq 0 \quad (u(x) \geq 0) \qquad \text{for } x \in B_{tr}'.
\]
\end{thm}

We will need the following lemma, which is along the lines of Corollary \ref{cor:nondegeneracy}.
\begin{lem}     \label{lem:steiner}
Let $u$ be a minimizer of \eqref{eq:jfunc-B1} such that $u|_{\partial B_1} = M$. Then $u$ is symmetric about the line $(0, \dots, 0, x_n)$ and 
\[
u(x',0) = f(|x'|)
\]
where $f$ is a nondecreasing function. Therefore, if the coincidence set is nonempty then the coincidence set $\{u=0\} = \overline{B}_{\rho}'$ for some $0 \leq \rho \leq 1.$
\end{lem}

\begin{proof}
Extend $u$ to be a function on the cube $Q$ with side length 2, by defining $u(x) = M$ for $x \notin B_1$. We now apply Steiner symmetrization (as defined in \cite[page 82]{bK85}) to the function $w = M - u$ on lines parallel to $\R^{n-1} \times \{0\} $. If we only consider $\{x \mid \ |x_n| > \epsilon\}$,  then $w$ is Lipschitz. Then by \cite[page 82]{bK85}, if we Steiner symmetrize $w$ to obtain $v$ deduce for each $\epsilon>0$:
\[
\int_{B_1 \cap \{|x_n|> \epsilon \} }{|\nabla u|^2}|x_n|^a = \int_{B_1 \cap \{|x_n|> \epsilon \} }{|\nabla w|^2|x_n|^a} 
\geq \ \int_{B_1 \cap \{|x_n|> \epsilon \} }{|\nabla v|^2|x_n|^a} 
\]
$v$ will have the same boundary values as $w$ on $\partial B_1$. 
Then by letting $\epsilon \to 0$ we obtain
\[
\int_{B_1}{|\nabla u|^2|x_n|^a} = \int_{B_1}{|\nabla w|^2|x_n|^a} \geq \ \int_{B_1}{|\nabla v|^2|x_n|^a}.
\]
Finally, we note that 
\[
\int_{B_{1}'}{\lambda_+ u} 
\]
is invariant under Steiner symmetrization. Since minimizers are unique, we see that our minimizer is Steiner symmetric  about the line $(0, \dots ,0, x_n)$ and $\{u=0\}$ is connected and centered at the origin.
\end{proof}

\begin{lem}  \label{lem:coincidence}
Let $u_{\epsilon}$ be the minimizer of \eqref{eq:jfunc-B1} subject to the boundary conditions $u_{\epsilon} \equiv \epsilon$ on $\partial B_1$. Then $u_{\epsilon} \equiv 0$ on $B_{\rho}'$ for some $\rho = \rho (\epsilon)>0$. Furthermore $\rho(\epsilon) \to 1$ as $\epsilon \to 0$. 
\end{lem}

\begin{proof}
It is clear that the minimizer $u_{\epsilon} \geq 0$. From the comparison principle $u_{\epsilon_1} \leq u_{\epsilon_2} $ if $\epsilon_1 \leq \epsilon_2$. Suppose by way of contradiction that $u_{\epsilon}(x',0) > 0 $ for all $\epsilon$ and for $x'$  in the ring 
\[
\mathcal{R} = \{(x',0) \in B_{1}' \mid 0 \leq r_1 < |x'| < r_2 \leq 1\}.
\]
By Lemma \ref{lem:normalbound}
\begin{equation*}  
\lim_{t\to 0} t^a\frac{\partial u_{\epsilon}}{\partial x_n}(x',t) = \lambda_+.
\end{equation*}
Thus 
$$
u_{\epsilon} - \tfrac{\lambda_+}{1-a}|x_n|^{1-a}
$$ will be $a$-harmonic in the tube 
$$\mathcal{T} = \{x \in B_1 \mid 0 \leq r_1 < |x'| < r_2 \leq 1\}.$$
Now, $u_{\epsilon} \to u \equiv 0$ uniformly, which implies that
$-\frac{\lambda_+}{1-a} |x_n|^{1-a}$ is $a$-harmonic in $\mathcal{T}$.
This is clearly a contradiction. Thus, there exists
$\epsilon$ and $y' \in \mathcal{R}$ such that $u_{\epsilon}(y',0) = 0$.  Then, by
Lemma \ref{lem:steiner}, $u_{\epsilon} \equiv 0$ on $B_{r_1}'$.  
\end{proof}
We are now able to prove the nondegeneracy result.
\begin{proof}[Proof of Theorem \ref{thm:nondegeneracy}]
First we note that by rescaling we only need to prove Theorem \ref{thm:nondegeneracy} on the unit ball $B_1$. Pick $0<t<1$. Lemma \ref{lem:coincidence} proves that there exists $\epsilon$ depending on $t$ and $\lambda_+$ such that $\{u_{\epsilon} =0\} = B_{t}'$. By the comparison principle if $u$ is a minimizer such that $u \leq u_{\epsilon} = \epsilon$ on $\partial B_1$, then $u \leq 0$ on $B_{t}'$.  
The case for which $u \geq -\epsilon$ is proven similarly.
\end{proof}
Theorem \ref{thm:nondegeneracy}  immediately implies the corollary below, which is the result we will be using later on.
\begin{cor} \label{cor:1/2growth}
If $u$ is a minimizer and $0 \in \Gamma^+$ $(0 \in \Gamma^-)$, then 
\begin{equation*}   
\sup_{\partial B_r}u \geq C r^{1-a}\qquad \left( \inf_{\partial B_r}u \leq -C r^{1-a} \right),
\end{equation*}
where $C$ depends only on $a$, $\lambda_\pm$ and $n$.
\end{cor}

\section{Blowups and global solutions}\label{sec:blow-glob-solut}

In this section we study the blowups and global solutions that will
be the main step towards proving Theorem~\ref{thm:main-free-bound}.

Let $u$ be a minimizer of \eqref{eq:jfunc-B1} and assume that $0\in
\Gamma^+\cup\Gamma^-$. For $r>0$ consider the rescalings
$$
u_r(x)=\frac{u(rx)}{r^{1-a}},\quad x\in B_{1/r}.
$$ 
From the scaling properties of $J_a$, it is easy to see that $u_r$ is
a minimizer of $J_a$ in $B_{1/r}$. By the optimal growth estimate we
will then have that
$$
\sup_{B_R}|u_r|\leq C R^{1-a},\quad (r<1/R) 
$$
when $a\geq 0$, so the family $\{u_r\}_{0<r<1}$ is locally bounded in
$\R^n$. Note that to have the same conclusion when $a<0$ we must assume
additionally that $|\nabla_{x'}u(0)|=0$. The local boundedness implies
boundedness in $W^{1,2}_\loc(\R^n, |x_n|^a)$, by the energy
inequality. Thus, over a sequence $r_k\to 0$ the rescalings
$u_{r_k}$ will converge to a certain $u_0$ weakly in
$W^{1,2}_\loc(\R^n, |x_n|^a)$. (This convergence is actually strong, as
we prove in Lemma~\ref{lem:str-conv} below.) Passing to a subsequence we may assume
that the convergence is strong in $L^2_\loc(\R^n, |x_n|^a)$ and
$L^2_\loc(\R^{n-1}\times\{0\})$. We call such $u_0$ a blowup of $u$ at
the origin. A standard argument also shows that
the blowup $u_0$ will be a minimizer of $J_a$ on any $U\Subset
\R^n$. Such minimizer we will call \emph{global minimizers} of $J_a$.

\medskip Our analysis of blowups starts with the following lemma on
the strong convergence of minimizers.

\begin{lem}[Strong convergence]\label{lem:str-conv} Let $\{u_k\}$ be a sequence of
  minimizers of $J_a$ in $D^+$ that converges weakly in
  $W^{1,2}(D^+,x_n^a)$ to some minimizer $u$. Then, over a
  subsequence, $u_k$ converges strongly to $u$ in $W^{1,2}(U^+,x_n^a)$
  for any $U\Subset D$.
\end{lem}
\begin{proof}  Take a test function $\eta\in C^\infty_0(D)$ such that
  \[
  0 \leq \eta \leq 1,\quad \eta \equiv 1 \text{ in a neighborhood of }
  \overline U.
  \]
Then, integrating by parts we have
\begin{align*}
 \int_{D^+}|\nabla (u_k-u)|^2\eta^2
 |x_n|^a&=\int_{D'}(f_k-f)(u_k-u)\eta^2\\
&\qquad-2\int_{D^+}
 (u_k-u)\eta \nabla\eta \nabla (u_k-u)|x_n|^a,
\end{align*}
where $f_k(x')=-\lim_{x_n\to 0+} x_n^a\partial_{x_n}u_k(x',x_n)$ and
$f(x')=-\lim_{x_n\to 0+} x_n^a\partial_{x_n}u(x',x_n)$ in the sense of
Lemma~\ref{lem:normalbound}, which also tells that $f_k$ and $f$ are uniformly bounded
by $\max\{\lambda_\pm\}$. Using that bound and applying the Young's
inequality to the second integral on the right-hand side in the usual
manner, we obtain that
\begin{align*}
\int_{U^+}|\nabla (u_k-u)|^2|x_n|^a&\leq C\int_{D'\cap K}
|u_k-u|+C\int_{D^+\cap K} (u_k-u)^2|x_n|^a,
\end{align*}
where $K=\supp\eta\Subset D$.
The proof now follows from the strong convergence of $u_k$ to $u$ in
$L^2(D'\cap K)$ and $L^2(D^+\cap K,|x_n|^a)$ by the compactness of the
trace operator and the Sobolev embedding.
\end{proof}

Having the strong convergence of minimizers combined with the
Weiss-type monotonicity formula, we immediately obtain the following
property of blowups.

\begin{lem}\label{lem:geneous}
Let $u$ be a minimizer of \eqref{eq:jfunc-B1} with $0 \in \Gamma^+\cup
\Gamma^-$. If
$a<0$ make the further assumption that $|\nabla_{x'} u(0)|=0$. Suppose
also that $u_{r_k}$ converges to a blowup $u_0$. Then $u_0$ is homogeneous of degree $(1-a)$.
\end{lem}
\begin{proof} We will use the following scaling property of the Weiss
  energy functional
$$
W(\rho r,u)=W(\rho,u_r).
$$
Now, by Lemma~\ref{lem:str-conv}, we may assume that $u_{r_k}\to u_0$
strongly in $W^{1,2}(U,|x_n|^a)$ for any $U\Subset\R^n$ and we can
pass in the limit in the Weiss energy functional to obtain
$$
W(\rho,u_0)=\lim_{k\to\infty} W(\rho, u_{r_k})=\lim_{k\to \infty}
W(\rho r_k, u)= W(0+,u)
$$
for any $\rho>0$. This implies that $W(\cdot,u_0)\equiv
\mathrm{const}=W(0+,u)$ and consequently that  $u_0$ is homogeneous of
degree $(1-a)$, by the second part of Theorem~\ref{thm:weiss}.
\end{proof}

\begin{lem}[Homogeneous global solutions] \label{lem:blowup}
Let $a\geq 0$ and let $u$ be a global minimizer of $J_a$ and assume $u$ is homogeneous of degree $(1-a)$. Then $u = c|x_n|^{1-a}$ for some constant $c$.
\end{lem}
\begin{proof}
Choose $c$ to be such that if $v = u - c|x_n|^{1-a}$, then 
\begin{equation}  \label{eq:avezero}
\int_{\partial B_1}{|x_n|^a v} = 0.
\end{equation}
We claim that 
\[
\frac{-\lambda_-}{(1-a)} \leq c \leq \frac{\lambda_+}{(1-a)}.
\]
The heuristic idea behind the claim is that if $c$ is too large then $L_{a} v \leq 0 $ in all of $B_1$, and so $v(0)>0$. To proceed proving this claim we  use
the definition of $f$ in as given in Lemma \ref{lem:normalbound}.  
Suppose that $c > \lambda_+ /(1-a)$.  Using that $\div(|x_n|^a \nabla v)=0$ off the thin spaces, we integrate by parts to obtain 
\begin{equation}  \label{eq:decreasing}
\int_{\partial B_r}{|x_n|^a v_{\nu}} = 2\int_{B_{r}'}{(f - c(1-a))}< 0,
\end{equation}
since $f \leq \lambda_+ \leq c(1-a)$. 
Also, for almost every $r$
\[
\frac{d}{dr} \left[ \frac{1}{r^{n-1+a}} \int_{\partial B_r}{|x_n|^a v}  \right] =
\int_{\partial B_r}{|x_n|^a v_{\nu}}.
\]
It follows that 
\[
\frac{1}{r^{n-1+a}} \int_{\partial B_r}{|x_n|^a v} \quad \text{is decreasing in } r.
\]
Then $v(0)>0$. This is a contradiction since $u(0)=0$, and consequently $v(0)=0$. A similar argument shows that $c \geq  -\lambda_- /(1-a)$.

Now we use integration by parts to obtain 
\[
\int_{B_r}{|x_n|^a |\nabla v|^2} = \int_{\partial B_r}{|x_n|^a v v_{\nu}} -2\int_{B_{r}'}{(f-c(1-a))v}.
\]
By homogeneity of $v$ we obtain that $v_{\nu}= (1-a)v/r $. We now turn
to simplifying the integral along the thin space. Note that $v=u$ on the thin
space. Also by Lemma \ref{lem:normalbound}, we may write $f(y')$
explicitly when $u(y',0) \neq 0$. Since we are multiplying by $u$ on
the thin space we only need to consider the case when $u(y',0)
\neq 0$. So our equality above becomes 
\begin{equation}  \label{eq:transition}
\begin{aligned}
\int_{B_r}{|x_n|^a |\nabla v|^2} &= \frac{1-a}{r}\int_{\partial B_r}{|x_n|^a v^2}  \\ 
&\qquad -2\int_{B_{r}'}{(\lambda_+ - c(1-a))u^+ + (\lambda_- + c(1-a))u^- }
\end{aligned}
\end{equation}
We then define $c_+ = (\lambda_+ - c(1-a))/2$ and $c_- = (\lambda_- +
c(1-a))/2$. Now $c_\pm \geq 0$ by our claim above. Also either $c_+
>0$ or $c_- >0$. If we move everything to the left hand side in
\eqref{eq:transition} and multiply by $r$,  we obtain 
\begin{equation}   \label{eq:homogeneous}
r \left(\int_{B_r}{|x_n|^a |\nabla v|^2} + 4\int_{B_r'}{(c_+ v^+ + c_- v^-)}\right) - (1-a)\int_{\partial B_r}{|x_n|^a v^2}=0.
\end{equation}
If $w$ is the $a$-harmonic replacement of $v$ in $B_r$, then $w(0)=0$
by \eqref{eq:avezero} and $w$ is even in $x_n$ since $v$ is even. Now
suppose $v$ is not equivalently zero. From Corollary~\ref{cor:weisscor} we obtain
$$
1\leq N(w,r)\leq N(v,r).
$$
Combining the above inequality with \eqref{eq:homogeneous} we have
\[
1\leq N(w,r)\leq N(v,r) \leq r\frac{\displaystyle\int_{B_r}{|x_n|^a |\nabla v|^2} + 4\int_{B_r'}{(c_+ v^+ + c_- v^-)} }{\displaystyle\int_{\partial B_r}{|x_n|^a v^2}} = 1-a.
\]
If $a>0$, then we obtain an immediate contradiction, and thus for $a>0$, $u \equiv c|x_n|^{1-a}$. If $a=0$, then necessarily 
\[
\int_{B_r'}{(c_+ v^+ + c_- v^-)} = 0.
\]
Then $v(x',0) \equiv 0$. By using an odd reflection in $x_n$ variable,
we obtain a harmonic function, homogeneous of degree $1$ in all of
$\R^n$. Thus, by the Liouville theorem we conclude that $v=c_3 x_n$ if
$x_n>0$ and by even symmetry then  $v = c_3|x_n|$ for all $x$. By \eqref{eq:avezero} we must have $c_3 =0$ and so $v \equiv 0$. 

\end{proof}

We conclude this section with the following generalization of
Lemma~\ref{lem:blowup}, which will be important in the next section
when studying the limits of rescalings of possibly different functions,
when we don't have the homogeneity, but can control the growth at infinity.

\begin{lem}[Global minimizers] \label{lem:global}
Let $u$ be a global minimizer of $J_a$ such that
$0\in\Gamma^+\cup\Gamma^-$ for $a\geq 0$. Assume
also that 
$$|u(x)|\leq C(1+|x|^{1-a}).
$$ 
Then $u = c|x_n|^{1-a}$ for some constant $c$.
\end{lem}
\begin{proof} We first perform a blow-up $u_{r_j}\to u_0$, and
  conclude by Lemma \ref{lem:blowup} that $u_0=c_0|x_n|^{1-a}$. Next,
  considering  again the usual rescalings
$$
u_R(x)=\frac{u(Rx)}{R^{1-a}}
$$
and noticing that $u_R$ are also locally uniformly bounded as
$R\to\infty$, we can extract a subsequence $R_j\to\infty$ such that
$u_{R_j}\to u_\infty$. Arguing precisely as in
Lemma~\ref{lem:geneous}, we obtain that $u_\infty$ is homogeneous and thus, by Lemma \ref{lem:blowup}, $u_\infty=c_\infty|x_n|^{1-a}$.

On the other hand, we can compute that $W(r, |x_n|^{1-a})\equiv0$, which implies
$$
0=W(1,u_0)=\lim_{r\to 0} W(r,u)\leq W(s,u)\leq \lim_{R\to \infty}
W(R,u)=W(1,u_\infty)=0,
$$
for any $s>0$. Hence, $W(u,s)\equiv 0$ and $u$ is homogeneous, and again by Lemma~\ref{lem:blowup}, $u=c|x_n|^{1-a}$.
\end{proof}

\section{Structure of the free boundary}\label{sec:struct-free-bound}

 In this section we prove
Theorem~\ref{thm:main-free-bound} in the following
more expanded form.

\begin{thm}[Structure of the free boundary]\label{thm:main2-expanded} Let $a\geq 0$ and let $u$
  be a bounded minimizer of \eqref{eq:jfunc-B1} such
  that $|u|\leq M$ in $B_1$. Then 
$$
\Gamma^+\cap \Gamma^-=\emptyset.
$$
More
  specifically, there is $c_0=c(a,M,\lambda_\pm, n)>0$ such that if $x_0\in
  \Gamma^+\cap B_{1/2}'$ then $B_{c_0}'(x_0)\cap
  \Gamma^-=\emptyset$. 
\end{thm}

\begin{proof} Let $x_0\in \Gamma_u^+\cap B_{1/2}$. By translation and rescaling, we
  may assume that $x_0=0$. Then, if the statement of the theorem
  fails, then there exists a sequence of minimizers $u_j$ of
  \eqref{eq:jfunc-B1} satisfying 
$$
|u_j|\leq M,\quad 0\in \Gamma^+_{u_j},\quad
|\nabla_{x'}u_j(0)|=0\quad\text{if }a<0.
$$
together with a sequence of points $x^j\in \Gamma^-_{u_j}$ such that
$r_j=|x^j|\to 0$. We will show that this is impossible. Let
$$
v_j(x)=(u_j)_{r_j}(x)=\frac{u_j(r_j x)}{r_j^{1-a}}.
$$ 
By Theorem~\ref{thm:holder} combined with Theorem~\ref{thm:main1-growth}, $\{v_j\}$ is uniformly bounded in
$C^{\alpha}(B_R)\cap W^{1,2}(B_R,|x_n|^a)$ for every ball $B_R$. Moreover, we know:
\begin{enumerate}
\item $v_j$ is a minimizer of $J_a$ in  $B_R\subset B_{1/r_j}$.
\item By Theorem~\ref{thm:main1-growth} we have the growth estimates 
$$
|v_j(x)|\leq C|x|^{1-a},\quad |x|\leq 1/(2r_j).
$$
\item Since $0\in \Gamma_{v_j}^+$ and $z^j=x^j/r_j\in
    \Gamma_{v_j}^-\cap B_1'$, by Theorem~\ref{thm:nondegeneracy} we have 
$$
\sup_{B_t}v_j\geq Ct^{1-a},\quad \inf_{B_t(z^j )}v_j\leq -Ct^{1-a}.
$$
\end{enumerate}
Consequently, we can extract a subsequence, again labeled $v_j$,
converging locally uniformly to $v_0$, a global minimizer of
$J_a$. Due to the nondegeneracy property (3) above, we also know that
$v_0$ has nontrivial positive and negative phases. On the other hand,
by (2) we also know that
$$
|v_0(x)|\leq C|x|^{1-a},\quad\text{for any }x\in\R^n
$$
and thus from Lemma \ref{lem:global} we conclude that $v_0 = c|x_n|^{1-a}$. Clearly, this is not possible.
\end{proof}

Theorem~\ref{thm:main-free-bound} now follows.

\begin{proof}[Proof of Theorem~\ref{thm:main-free-bound}] This is an
  immediate consequence of the local boundedness (see 
  Corollary~\ref{cor:loc-bdd}), Theorem~\ref{thm:main2-expanded}, and
  a simple covering argument by balls. 
\end{proof}

We conclude this section and the paper by providing an example showing
that when $a<0$ it is possible for the two phases to meet at a point where the thin gradient vanishes. This is
done by constructing a minimizer whose Weiss energy at a point is negative. 

\begin{xmp}[The two phases can meet where the gradient vanishes, when $a<0$] \label{xmp:canmeet}

For simplicity in our construction we will assume $\lambda_+ =
\lambda_-=\lambda$. Fix some $i\in \{1,\ldots, n-1\}$ and let $\ell(x)=Mx_i$.
Then consider the minimizer of the functional $J_a$ with
boundary values given by $u = \ell$ on $\partial B_1$. We claim that there is an $M>0$ such that $0\in \Gamma^+\cap \Gamma^-$ and $\nabla_{x'} u(0)=0$. By
symmetry of the boundary values of $u$ on $\partial B_1$ combined with
our choice of $\lambda_+ = \lambda_-$ and the fact that minimizers are
unique, it is clear that $u(0)=0$. Now we look at the Weiss energy of
$\ell$ in $B_1$. By homogeneity of $\ell$ 
\[
W(1,\ell)= M \int_{B_{1}'}{\lambda|x_i|} + aM^2\int_{\partial B_1}{x_i^2 |x_n|^a},
\] 
and we see that for large enough $M$ (since $a<0$), $W(1,\ell)
<0$. Since $u$ and $\ell$ coincide on $\dd B_1$, this implies that $W(1,u)<0$. Now we argue that $0\in \Gamma_u^+\cup\Gamma_u^-$. Indeed, otherwise 
then $u \equiv 0$ in $B_{\rho}'$ for some small $\rho$ (recall
$u(0)=0$). If we reflect $u$ by odd reflection across the thin space
we obtain $\tilde{u}$ which is $a$-harmonic in the solid ball
$B_{\rho}$ so that $W(\rho, u)=W(\rho, \tilde{u})$. By Corollary~\ref{cor:weisscor}, $W(\rho, \tilde{u}) \geq 0$, contradicting
$W(\rho,u)\leq W(1,u)<0$. Thus, $0 \in
\Gamma_u^+\cup\Gamma_u^-$ and by symmetry and uniqueness, $0 \in
\Gamma_u^+\cap\Gamma_u^-$.

If $|\nabla_{x'}u(0)|=0$, then the claim is proved. If not, $|\nabla_{x'}u(0)|\neq 0$ for $M$ large enough. In this case let 
$$
\overline M = \inf\{M:\,|\nabla_{x'}u(0)|\neq 0\}. 
$$
Non-degeneracy (cf. Theorem \ref{thm:nondegeneracy}) implies $u\equiv 0$ near the origin for $M$ small enough. Hence, $|\nabla_{x'}u(0)|=0$ for $M$ small enough and thus $\overline M>0$. Interior $C^{1,\alpha}$-convergence (cf. Theorem \ref{thm:main1-expanded}) implies $|\nabla_{x'}u_{\overline M}(0)|=0$. We claim that $0\in \Gamma_{u_{\overline M}}$. Suppose towards a contradiction that $0\not\in \Gamma_{u_{\overline M}}$.

By symmetry and uniqueness of minimizers, $u(0)=0$ and $0\in \Gamma_{u}^+$ ($0\in \Gamma_{u}^-$) implies $0\in \Gamma_{u}^-$ ($0\in \Gamma_{u}^+)$. For $M>\overline M$, the fact that $0\in \Gamma^\pm_{u}$ implies
\begin{equation}
\label{eq:M-non-deg}
\sup_{\dd B_r} u\geq Cr^{1-a},\quad \inf_{\dd B_r} u\leq -Cr^{1-a}
\end{equation}
for some $C>0$ and for $0<r<1$. By the uniform convergence, this holds true also for $u_{\overline M}$. If $0\not\in \Gamma_{u_{\overline M}}$ then $u_{\overline M}\equiv 0$ in $B_\rho'$ for some $\rho>0$. Now we perform a blowup at the origin of $u_{\overline M}$ and obtain $u_0$, which is homogeneous of degree $1-a$, by Lemma \ref{lem:geneous}. By the scaling invariance, $u_0$ satisfies \eqref{eq:M-non-deg}. Since $u\equiv 0$ in $B_\rho'$ for some $\rho>0$, then $u_0\equiv 0$ in $\R^{n-1}$. If we extend $u_0$ by odd reflection in the $x_n$-variable, we obtain an $a$-harmonic function in the whole space. Since $u_0$ is homogeneous of degree $1-a$, we must have $u_0=c|x_n|^{1-a}$ (cf. Lemma \ref{lem:thin-grad-est} for instance), contradicting \eqref{eq:M-non-deg}.

\end{xmp}

\appendix

\section{Proof of Weiss monotonicity formula}

\begin{proof}[Proof of Theorem \ref{thm:weiss}]
 The proof is very similar to that of Theorem 4.3 in \cite{a11} and we
 omit the details that are derived identically.

Using that $u$ is a minimizer of \eqref{eq:jfunc-B1}, we obtain that
for a.e.\ $r\in(0,1)$ we have the equality

\begin{equation}   \label{eq:almgrenmin}
\begin{aligned}
0 & = (n-2+a)\int_{B_r}{|x_n|^a |\nabla u|^2} - r \int_{\partial B_r}{|x_n|^a \left( |\nabla u|^2 -2u_{\nu}^2 \right)} \\
  & \qquad + 4(n-1)\int_{B_{r}'}{(\lambda_+ u^+ + \lambda_- u^-)}
    -4r \int_{\partial B_{r}'}{(\lambda_+ u^+ + \lambda_- u^-)} \\
  & = (n-1) \int_{B_r}{|x_n|^a|\nabla u|^2} - r \int_{\partial B_r}{|x_n|^a|\nabla u|^2} \\
  & \qquad + 4(n-1) \int_{B_{r}'}{(\lambda_+ u^+ + \lambda_- u^-)}
    -4r \int_{\partial B_{r}'}{(\lambda_+ u^+ + \lambda_- u^-)} \\
  & \qquad -(1-a) \int_{B_r}{|x_n|^a|\nabla u|^2} + 2r \int_{\partial B_r}{|x_n|^a u_{\nu}^2}.
\end{aligned}
\end{equation}
By Lemma~\ref{lem:normalbound}, 
\[
\int_{B_r}{|x_n|^a |\nabla u|^2} = \int_{\partial B_r}{|x_n|^a u u_{\nu}} + 2 \int_{B_{r}'}{(\lambda_+ u^+ + \lambda_- u^-)},
\]
and hence
\begin{alignat*}{2}0 & = (n-a) \int_{B_r}{|x_n|^a |\nabla u|^2} - r \int_{\partial B_r}{|x_n|^a|\nabla u|^2} \\
  & \qquad + 4(n-a) \int_{B_{r}'}{(\lambda_+ u^+ + \lambda_- u^-)}
    -4r \int_{\partial B_{r}'}{(\lambda_+ u^+ + \lambda_- u^-)} \\
  & \qquad - 2(1-a)\int_{\partial B_r}{|x_n|^a u \cdot u_{\nu}} + 2r \int_{\partial B_r}{|x_n|^a u_{\nu}^2} 
\end{alignat*}
Now multiply both sides of the equation by $-r^{-n+1-a}$ to obtain
that for a.e.\ $r\in (0,1)$
\begin{alignat*}{2}
0 & = \left[\frac{1}{r^{n-a}} \int_{B_r}{|x_n|^a|\nabla u|^2} \right]'
    + \left[\frac{4}{r^{n-a}} \int_{B_r '}{(\lambda_+ u^+ + \lambda_- u^-)} \right]' \\
  & \qquad  
   - \frac{2}{r^{n-a}} \int_{\partial B_r}{|x_n|^a\left(\frac{(1-a)u u_\nu}{r} -  u_{\nu}^2 \right)}.
\end{alignat*}
As in  \cite{a11}, we also have that for almost every $r$
\begin{equation}      \label{eq:der}
\frac{\d}{\d r}\left[\frac{1-a}{r^{n+1-a}} \int_{\partial B_r}{|x_n|^a u^2} \right] = 
\frac{2}{r^{n-a}} \int_{\partial B_r}{|x_n|^a \left( \frac{(1-a)u u_{\nu}}{r} - \frac{(1-a)^2u^2}{r^2}  \right)}.
\end{equation}
We then add and subtract the piece from \eqref{eq:der} to obtain for almost every $r$
\begin{alignat*}{2}
0 & = \left[\frac{1}{r^{n-a}} \int_{B_r}{|x_n|^a|\nabla u|^2} \right]'
    + \left[\frac{4}{r^{n-a}} \int_{B_r '}{(\lambda_+ u^+ + \lambda_- u^-)} \right]' \\
  & \qquad - \left[\frac{1-a}{r^{n+1-a}}  \int_{\partial B_r}{|x_n|^a u^2} \right]' 
    - \frac{2}{r^{n-a}} \int_{\partial B_r}{|x_n|^a \left(\frac{(1-a)u}{r} -  u_{\nu} \right)^2}
\end{alignat*}
Thus, $W' \geq 0$, and $W' = 0$ on the interval $r_1 <r< r_2$ if and only if $u$ is homogeneous of degree $2s=(1-a)$ on the ring $r_1 < |x| < r_2$.
\end{proof}

\bibliographystyle{amsalpha}
\bibliography{ref}  
\end{document}